\documentclass[12pt, twosided]{article}
\usepackage{fullpage}
\usepackage[T1]{fontenc}
\usepackage[mathlines]{lineno} 
\usepackage{amsmath, amsthm, amsfonts, amssymb}
\usepackage{graphicx}
\usepackage{color}
\usepackage{lineno}
\usepackage{thm-restate}
\usepackage{enumitem}
\usepackage[dvipsnames]{xcolor}
\usepackage[utf8]{inputenc}
\usepackage[english]{babel}
\usepackage{cite}
\usepackage[title]{appendix}
\usepackage{hyperref}

\usepackage{etoolbox} 

\newcommand*\linenomathpatch[1]{%
	\cspreto{#1}{\linenomath}%
	\cspreto{#1*}{\linenomath}%
	\csappto{end#1}{\endlinenomath}%
	\csappto{end#1*}{\endlinenomath}%
}
\newcommand*\linenomathpatchAMS[1]{%
	\cspreto{#1}{\linenomathAMS}%
	\cspreto{#1*}{\linenomathAMS}%
	\csappto{end#1}{\endlinenomath}%
	\csappto{end#1*}{\endlinenomath}%
}

\expandafter\ifx\linenomath\linenomathWithnumbers
\let\linenomathAMS\linenomathWithnumbers
\patchcmd\linenomathAMS{\advance\postdisplaypenalty\linenopenalty}{}{}{}
\else
\let\linenomathAMS\linenomathNonumbers
\fi

\linenomathpatch{equation}
\linenomathpatchAMS{gather}
\linenomathpatchAMS{multline}
\linenomathpatchAMS{align}
\linenomathpatchAMS{alignat}
\linenomathpatchAMS{flalign}

\newtheorem{theorem}{Theorem}
\newtheorem{thm}{Theorem}
\newtheorem{remark}{Remark}

\newtheorem{cor}[thm]{Corollary}
\newtheorem{lemma}[theorem]{Lemma}

\newtheorem{proposition}{Proposition}
\newtheorem*{prop}{Proposition}

\theoremstyle{definition}

\theoremstyle{remark}

\title{Hamiltonicity of Token Graphs of some Join Graphs}
\author{Luis Enrique Adame\thanks{Unidad Acad\'emica de Matem\'aticas, Universidad Aut\'onoma de Zacatecas, Zacatecas, Mexico. \texttt{l\_e\_a\_m\_@hotmail.com, luismanuel.rivera@gmail.com}} \and Luis Manuel Rivera\footnotemark[1] \and Ana Laura Trujillo-Negrete\thanks{Departamento de Matem\'aticas, Cinvestav, CDMX, Mexico, \texttt{ltrujillo@math.cinvestav.mx}}
}
\date{\today}

\begin{document}
\maketitle


\abstract{Let $G$ be a simple graph of order $n$ with vertex set $V(G)$ and edge set $E(G)$, and let $k$ be an integer such that $1\leq k\leq n-1$.  
	The $k$-token graph $G^{\{k\}}$ of $G$ is the graph whose vertices are the $k$-subsets of $V(G)$, where two vertices $A$ and $B$ are adjacent in $G^{\{k\}}$ whenever their symmetric difference $A\triangle B$, defined as $(A\setminus B)\cup (B\setminus A)$, is a pair $\{a,b\}$ of adjacent vertices in $G$.  In this paper we study the Hamiltonicity of the $k$-token graphs  of some join graphs. We provide an infinite family of graphs, containing Hamiltonian and non-Hamiltonian graphs, for which their $k$-token graphs are Hamiltonian. Our result provides, to our knowledge, the first family of non-Hamiltonian graphs for which it is proven the Hamiltonicity of their $k$-token graphs, for any $2<k<n-2$.}
\\

{\it Keywords:}  Token graphs, Hamiltonicity, fan graphs.\\
{\it AMS Subject Classification Numbers:}    05C76, 05C45.

\section{Introduction}

Let $G$ be a simple graph of order $n \geq 2$ with vertex set $V(G)$ and edge set $E(G)$, and let $k$ be an integer such that $1\leq k\leq n-1$. The \emph{$k$-token graph} $G^{\{k\}}$ of $G$ is the graph whose vertices are all the $k$-subsets of $V(G)$, where two vertices $A$ and $B$ are adjacent in $F_k(G)$ whenever their symmetric difference $A\triangle B$, defined as $(A\setminus B)\cup (B\setminus A)$, is a pair $\{a,b\}$ of adjacent vertices in $G$. See an example in Figure~\ref{fig:definition-example}. 
The token graphs are a generalization of the Johnson graphs~\cite{Jo}. For $n>k\ge 1$, the Johnson graph $J(n,k)$ is the graph whose vertices are all the $k$-subsets of the set $\{1,2,\dots,n\}$, where two $k$-subsets are adjacent in $J(n,k)$ if they intersect in $k-1$ elements. Thus, $J(n,k)=F_k(K_n)$. Token graphs and $k$-uniform hypergraphs are also related as follows: consider the complete $k$-uniform hypergraph $H$ on $\{1,2,\dots,n\}$. The $t$-line graph of $H$ is the graph whose vertices are the edges of $H$, two being joined if the edges they represent intersect in at least $t$ elements. Then, the $k$-token graph of $K_n$ is isomorphic to the ($k-1$)-line graph of $H$.

\begin{figure}[t]
	\centering
	\includegraphics[width=0.6\textwidth]{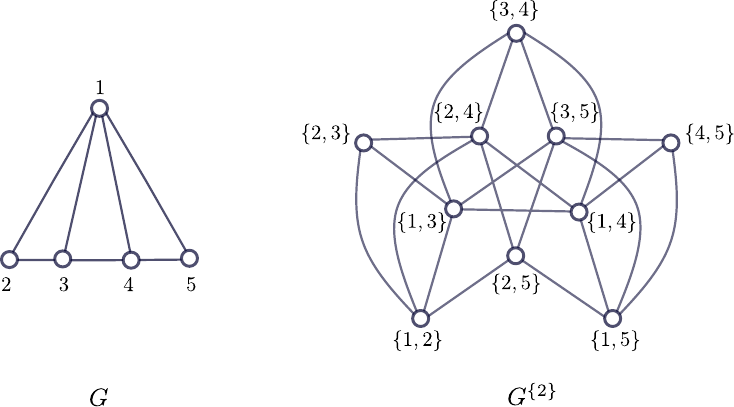}
	\caption{\small A graph $G$ and its $2$-token graph $G^{\{2\}}$.}
	\label{fig:definition-example}
\end{figure}

The $2$-token graphs are usually called  \emph{double vertex graphs} and the study of its combinatorial properties began in the 90's with the works of Alavi and its coauthors, see, e.g., \cite{alavi1, alavi2, alavi3, alavi4}, where they studied the connectivity, planarity, regularity and Hamiltonicity of some of such  graphs. Later, Zhu et al.~\cite{zhu} generalized this concept to the $k$-tuple vertex graphs, that are in fact the $k$-token graphs. In 2002, Rudolph \cite{aude, rudolph} redefined the token graphs, calling them {\it symmetric powers of graphs}, with the aim to study the graph isomorphism problem and for its possible applications to quantum mechanics. Several authors have continued with the study of the possible applications of  token graphs in physics (see. e.g., \cite{fisch2, fisch, ouy}). 

In 2012, Fabila-Monroy et al.~\cite{FFHH} reintroduced the concept of $k$-token graph of $G$ with the following interpretation. Consider $k$ indistinguishable tokens and placed them on the vertices of $G$ (at most one token per vertex), where each token can be slid from one vertex to another along the edges of $G$. Define a new graph whose vertices are all the possible token configurations, with two of such configurations being adjacent if one configuration can be reached from the other by sliding a token along an edge of $G$. This new graph is isomorphic to the $k$-token graph $G^{\{k\}}$ of $G$. In Figure~\ref{fig:tokens} is depicted the graph $G^{\{2\}}$ of Figure~\ref{fig:definition-example} but with this interpretation. 
Fabila-Monroy et al.~\cite{FFHH} began 
	a systematic study of some properties of $G^{\{k\}}$ such as connectivity, diameter, clique number, chromatic number and Hamiltonicity. This line of research has been continued by different authors, see, e.g.,  
	\cite{dealba2, token2, dalfo, deepa, deepa2, ruyanalea, soto, leancri, leatrujillo}.  

\begin{figure}[h]
	\centering
	\includegraphics[width=0.5\textwidth]{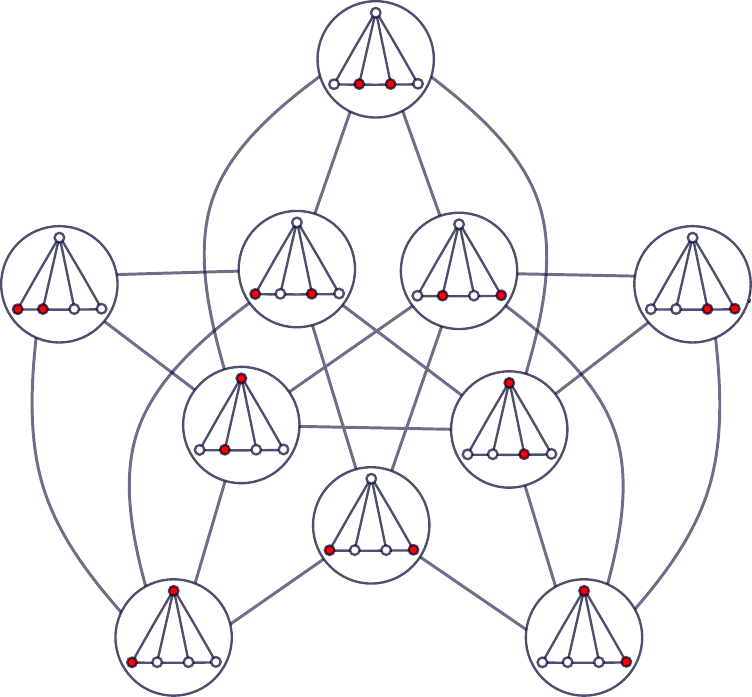}
	\caption{A $2$-token graph seen as the model of tokens moving on the graph.}
	\label{fig:tokens}
\end{figure}

\subsection{Hamiltonicity in token graphs} 

A  \emph{Hamiltonian cycle} of a graph $G$ is a cycle containing each vertex of $G$ exactly once. A graph $G$ is \emph{Hamiltonian} if it contains a Hamiltonian cycle. The Hamiltonian problem is to determine whether  a graph is Hamiltonian. Although the problem has been widely studied, there is no known non-trivial characterization of Hamiltonian graphs. On the positive side, several necessary and sufficient conditions have been discovered. Particularly, the Hamiltonicity problem for bipartite graphs has been studied in depth, see, e.g., \cite{shang}. 
Regarding the computational complexity of the problem, it is well-known that the Hamiltonian problem is NP-Complete \cite{np-complete-problems}.

It is well known that the Hamiltonicity of $G$ does not imply the Hamiltonicity of $G^{\{k\}}$. 
For example, Fabila-Monroy et al.~\cite{FFHH} showed that if $k$ is even, then $K_{m, m}^{\{k\}}$ is non-Hamiltonian. 
An easier, traditional example is the case of a cycle graph; it is known (see, e.g.~\cite{alavi4}) that if $n=4$ or $n \geq 6$, then $C_n^{\{2\}}$ is not Hamiltonian. On the other hand, 
there exist non-Hamiltonian graphs for which their double vertex graph is Hamiltonian, a simple example is the graph $K_{1,3}$, for which $K_{1,3}^{\{2\}}\simeq C_6$, and so $K_{1,3}^{\{2\}}$ is Hamiltonian. Some results about the Hamiltonicity of some double vertex graphs can be found, for example, in the survey of Alavi et. al.~\cite{alavi3}.

To our knowledge, for any integer $k>2$, all the graphs $G$ for which the Hamiltonicity of $G^{\{k\}}$ (resp. the existence of a Hamiltonian path in $G^{\{k\}}$) has been proven  are Hamiltonian (resp. contain a Hamiltonian path). This fact motivates us to search for a family of non-Hamiltonian graphs (and without Hamiltonian paths) with Hamiltonian $k$-token graphs, where $2<k<n$. 
 
\subsection{Basic definitions and results}
In order to formulate our results, we need the following definitions. 
Given two disjoint graphs $G$ and $H$, the \textit{join graph} $G+H$ 
of $G$ and $H$ is the graph whose vertex set is $V(G)\cup V(H)$ and its edge set is $E(G)\cup E(H)\cup \{xy:x\in G\text{ and }y\in H\}$. The \emph{generalized fan graph}, $F_{m,n}$, or simply \emph{fan graph}, is the join graph $F_{m,n}=\overline{K_m} +P_n$, where $\overline{K_m}$ denotes the graph of $m$ isolated vertices and $P_n$ denotes the path graph of $n$ vertices. The graph $G$ of Figure~\ref{fig:definition-example} is isomorphic to $F_{1, 4}$.

The join of graphs is a binary operation that has been widely studied in different contexts (see. e. g., \cite{alhevaz, carba, juan}), including its possible applications, as in  quantum information \cite{angeles}. As was mentioned before, the Hamiltonicity of bipartite graphs has been studied in depth, in some sense, and hence we are interested in the Hamiltonicity of non-bipartite $k$-token graphs $G^{\{k\}}$. Due to a result of Fabila-Monroy et al. \cite{FFHH}, such graphs come from non-bipartite graphs G. In this paper we are interested in the join of two graphs in which one of them contains a Hamiltonian path. These graphs are non-bipartite and may be Hamiltonian or not. 

In 2018, Rivera and Trujillo-Negrete~\cite{rive-tru} showed the Hamiltonicity of $F_{1,n}^{\{k\}}$, for any integers $k$ and $n$ with $1<k<n$. In this paper,  we study the Hamiltonicity of the $k$-token graphs of fan graphs $F_{m,n}$, for $m>1$. The family of fan graphs $F_{m,n}$ contains an infinite number of graphs containing no Hamiltonian paths, precisely, when $m>n+1$.
 
Our main result for the case $k=2$ is the following

\begin{theorem}
	\label{thm:main1}
	The double vertex graph of $F_{m,n}$ is Hamiltonian if and only if $n\geq 2$ and $1\leq m \leq 2\,n$, or $n=1$ and $m=3$. 
\end{theorem}

For the general case $k \geq 2$, our main result is the following.
\begin{restatable}{theorem}{kgeneral}
	\label{thm:general-case-main}
	Let $k, n, m$ be positive integers such that  $2\le k\le n$ and $1\leq m\leq 2n$. Then $F_{m, n}^{\{k\}}$ is Hamiltonian. 
\end{restatable}

It is easy to see that if $H$ is a spanning subgraph of $G$, then $F_k(H)$ is a spanning subgraph of $F_k(G)$, for any integer $k$ with $1\le k\le n-1$. This property and Theorem~\ref{thm:general-case-main} implies the following:  

\begin{cor}
	\label{cor:main1}
	Let $k, n, m$ be positive integers such that  $2\le k\le n$ and $1\leq m\leq 2n$. Let $G_1$ and $G_2$ be two graphs
	of order $m$ and $n$, respectively, such that $G_2$ has a Hamiltonian path. Let $G=G_1+G_2$. Then $G^{\{k\}}$ is Hamiltonian. 
\end{cor}

We point out that this corollary provides our desired family of non-Hamiltonian graphs (and without  Hamiltonian paths) with Hamiltonian $k$-token graphs, 
for example, if $n+1<m\le 2n$, then $F_{m,n}$ contains no Hamiltonian path, while, as we will show in Theorem~\ref{thm:general-case-main}, $G^{\{k\}}$ is Hamiltonian, for any $2\le k\le n+m-2$.  

Finally, we give some notation that we will use in our proofs. Let $V(P_n):=\{v_1,\ldots,v_n\}$ 
and $V(\overline{K_m}):=\{w_1,\ldots, w_m\}$, so  
$V(F_{m,n})=\{v_1,\ldots,v_n,w_1,\ldots,w_m\}$. 
For a path $T=a_1 a_2 \dots a_{l-1}a_{l}$, we denote by $\overleftarrow{T}$ the reverse path $a_la_{l-1}\dots a_2a_1$. 
For a graph $G$, we denote by $\mu(G)$ the number of connected components of $G$. 

The rest of the paper is organized as follows. In Section~\ref{sec:double} we
present the proof of Theorem~\ref{thm:main1} and in Section~\ref{sec:main} the proof of Theorem~\ref{thm:general-case-main}. Our strategy to prove these results is the following: for $k=2$, we show explicit Hamiltonian cycles, and for $k>2$, we use induction on $m$; we often use the interpretation of the $k$-token graph of $G$ as the model of $k$ tokens moving along the edges of $G$. In Section~\ref{app:Gray} we give the relationship between the Hamiltonicity of token graphs and Gray codes for combinations. Finally, in Section~\ref{sec:concl} we present the conclusions and suggest three open problems. 

\section{Proof of Theorem~\ref{thm:main1}}
\label{sec:double}

This section is devoted to the proof of Theorem~\ref{thm:main1}. For the case $1\leq m\leq 2n$, we construct an explicit Hamiltonian cycle in $F_{m,n}^{\{2\}}$, 
in which the vertices $\{w_1,v_1\}$ and $\{v_1,v_2\}$ are adjacent (this fact will be used in the proof of Theorem~\ref{thm:general-case-main} for the case $k>2$).  

If $n=1$ then $F_{m, n} \simeq K_{1,m}$, and it is known that $K_{1,m}^{\{2\}}$ is Hamiltonian if and only if $m=3$ 
(see, e.g., Proposition 5 in~\cite{alavi3}). From now on, assume $n\geq 2$.  We distinguish four cases: either $m=1$, $m=2n$, 
$1<m<2n$ or $m>2n$. 

	\begin{figure}[t]
		\centering 
		\includegraphics[width=0.8\textwidth]{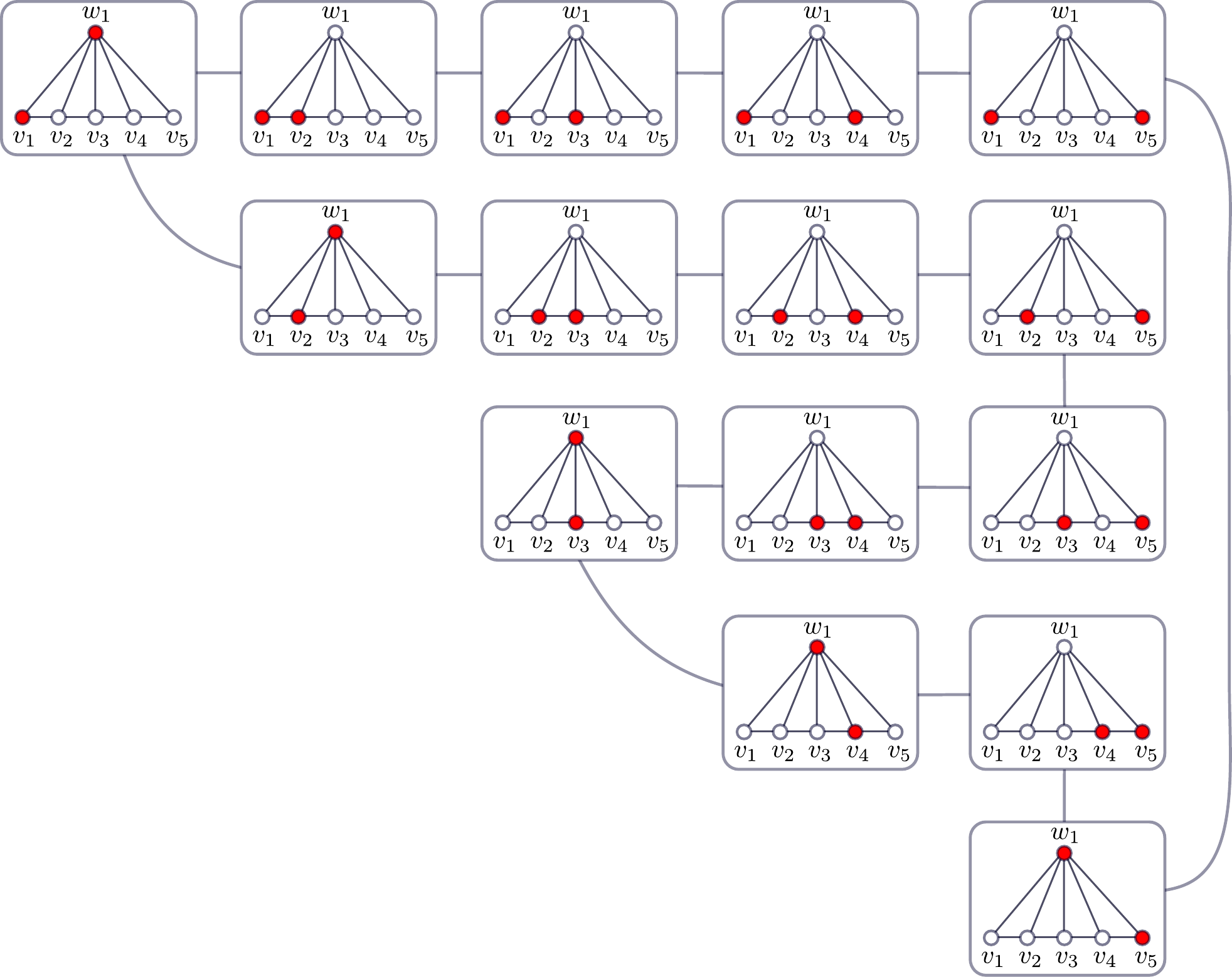}
		\caption{Hamiltonian cycle $C$ in the $2$-token graph of $F_{1,5}$.}
		\label{fig:construction1}
	\end{figure}

\begin{itemize}
	\item {\bf Case} $\boldsymbol{m=1.}$ 	
	
	For $n=2$ we have $F_{1,2}^{\{2\}}\simeq F_{1,2}$, and so $F_{1,2}^{\{2\}}$ is Hamiltonian. 
	Now we work the case $n\geq 3$. For $1\leq i< n$ let 
	\[T_i:=\{v_i,w_1\}\{v_i,v_{i+1}\}\{v_i, v_{i+2}\} \dots \{v_i,v_n\}\]
	and let $T_n:=\{v_n,w_1\}$.
	It is clear that every $T_i$ is a path in $F_{1,n}^{\{2\}}$ and that $\{T_1, \dots, T_n\}$ 
	is a partition of $V\left(F_{1,n}^{\{2\}}\right)$. 
	
	Let 
	\[
	C:=
	\begin{cases}
	\overleftarrow{T_1}\,T_2\overleftarrow{T_3}\,T_4\dots \overleftarrow{T_{n-1}}\,T_n\{v_1, v_n\} & \text{ if $n$ is even,} \\
	\overleftarrow{T_1}\,T_2\overleftarrow{T_3}\,T_4\,\dots \,T_{n-1}\,\overleftarrow{T_{n}}\,\{v_1, v_n\} & \text{ if $n$ is odd.} \\
	\end{cases} 
	\]
	We are going to show that $C$ is a Hamiltonian cycle of $F_{1,n}^{\{2\}}$. Suppose that $n$ is even, so
		\begin{equation*}
			C=\underbrace{\{v_1, v_n\} \dots \{v_1, w_1\}}_{\overleftarrow{T_1}}\underbrace{\{v_2, w_1\} \dots \{v_2, v_n\}}_{T_2}\,\dots\,\underbrace{\{v_n, w_1\}}_{T_n}\{v_1, v_n\}.
	\end{equation*}
	
	For $i$ odd, the final
	vertex of $\overleftarrow{T_i}$ is $\{v_i,w_1\}$, while the initial vertex of $T_{i+1}$ is $\{v_{i+1},w_1\}$, and
	since these two vertices are adjacent in $F_{1,n}^{\{2\}}$, the concatenation $\overleftarrow{T_i}\,T_{i+1}$ corresponds to 
	a path in $F_{1,n}^{\{2\}}$. Similarly, for $i$ even, the final vertex of $T_i$ is $\{v_i,v_n\}$ while the initial
	vertex of $\overleftarrow{T_{i+1}}$ is $\{v_{i+1},v_n\}$, so again, the concatenation $T_i\,\overleftarrow{T_{i+1}}$
	corresponds to a path in $F_{1,n}^{\{2\}}$. Also note that the unique vertex of $T_n$ is $\{v_n,w_1\}$, which is adjacent
	to $\{v_1,v_n\}$. As the initial vertex of  $\overleftarrow{T_1}$ is $\{v_1,v_n\}$, we have that $C$ is a cycle in $F_{1,n}^{\{2\}}$. 
As an example, in Figure~\ref{fig:construction1} we show the Hamiltonian cycle $C$ in $F_{1,5}^{\{2\}}$, which is constructed as above.

The proof for $n$ odd is analogous. 

	
	\item {\bf Case} $\boldsymbol{m=2n.}$

	Let $C$ be the cycle defined in the previous case depending on the parity of $n$.
	Let 
	\[
	P_1:=\{v_n,w_1\} \xrightarrow{C} \{v_1,v_n\}
	\]   
	be the path obtained from $C$ by deleting the edge between $\{v_n, w_1\}$ and
	$\{v_1, v_n\}$. 
	For $1< i\leq n$ let
	\[
	\begin{split}
	P_i:=&\{w_i,v_n\}\{w_i,w_1\}\{w_i,v_{n-1}\}\{w_i,w_{i+(n-1)}\}\{w_i,v_{n-2}\}\{w_i,w_{i+(n-2)}\}\{w_i,v_{n-3}\}\\&\{w_i,w_{i+(n-3)}\}\ldots 
	\{w_i,v_2\}\{w_i,w_{i+2}\}\{w_i,v_1\}\{w_i,w_{i+1}\}.
	\end{split}
	\]
	We can observe that after $\{w_i, w_1\}$, the vertices in the path $P_i$ follows the  pattern $\{w_{i}, v_j\}\{w_{i}, w_{i+j}\}$, from $j=n-1$ to $1$.
	For $n+1\leq i\leq 2n$ let
	\[
	\begin{split}
	P_i:=&\{w_i,v_n\}\{w_i,w_{i+n}\}\{w_i,v_{n-1}\}\{w_i,w_{i+(n-1)}\}\{w_i,v_{n-2}\}\{w_i,w_{i+(n-2)}\}\ldots\\
	& \{w_i,v_2\}\{w_i,w_{i+2}\}\{w_i,v_1\}\{w_i,w_{i+1}\},
	\end{split}
	\]
	where the sums are taken mod $2n$ with the convention that $2n\pmod{2n}=2n$. In this case, the vertices  in $P_i$ after $\{w_i,w_{i+n}\}$ follow the  pattern $\{w_{i}, v_j\}\{w_{i}, w_{i+j}\}$, from $j=n-1$ to $1$.
	
Let
		\[
		C_2:=P_1\,P_2\,\ldots\,P_{2n}\{v_n,w_1\}.
		\]
		Let us show that $C_2$ is a Hamiltonian cycle of $F_{m,n}^{\{2\}}$.
	First we show that \[\{V(P_1),\dots,V\left(P_{2n}\right)\}\] is a partition of $V\left(F_{m,n}^{\{2\}}\right)$.  
	\begin{itemize}
		\item $\{v_i, v_j\}$ belongs to $P_1$, for any $i,j\in [n]$ with $i\neq j$.
		\item $\{w_i,v_j\}$ belongs to $P_i$, for any $i\in [m]$ and $j\in [n]$.
		\item $\{w_i, w_1\}$ belongs to $P_i$, for any $i\in [m]$ with $i\neq 1$. 
		\item Consider now the vertices of type $\{w_i,w_j\}$, for $1<i<j\leq n$, 
		\begin{itemize}
			\item $\{w_i,w_j\}$ belongs to $P_i$, for any $1<i\leq n$ and $i<j\leq i+n-1$.
			\item $\{w_i,w_j\}$ belongs to $P_j$, for any $1<i\leq n$ and $i+n-1< j\leq 2n$.
			\item $\{w_i,w_j\}$ belongs to $P_i$, for any $n< i< 2n$ and $i<j\leq 2n$.
		\end{itemize}
	\end{itemize}
Thus, $\{V(P_1),\dots,V\left(P_{2n}\right)\}$ is a partition of $V\left(F_{m,n}^{\{2\}}\right)$. Next we show that $C_2$ is a cycle. 
	We observe that
	\begin{enumerate}[label=(\arabic*)]
		\item $P_i$ induces a path in $F_{m,n}^{\{2\}}$, for each $i\in [2n]$; 
		\item the final vertex of $P_1$ is $\{v_1,v_n\}$, while the initial vertex of $P_2$ is $\{w_2, v_n\}$, and these
	two vertices are adjacent in $F_{m,n}^{\{2\}}$; 
	\item for $i$ with $1<i<2n$, the final vertex of $P_i$ is $\{w_i,w_{i+1}\}$ while
	the initial vertex of $P_{i+1}$ is $\{w_{i+1},v_n\}$, and these two vertices are adjacent in $F_{m,n}^{\{2\}}$; and
	\item the final vertex of $P_{2n}$ is $\{w_1,w_n\}$ while the initial vertex of $P_1$ is $\{v_n,w_1\}$, and
	these two vertices are adjacent in $F_{m,n}^{\{2\}}$. 
\end{enumerate}
Statements (1)--(4) together imply that $C_2$ is a cycle in $F_{m,n}^{\{2\}}$. Thus, $C_2$ is a Hamiltonian cycle of $F_{m,n}^{\{2\}}$. Note that the vertices $\{w_1,v_1\}$ and $\{v_1,v_2\}$ are adjacent in $C_2$, since they are adjacent in $P_1$.

As an example, in Figure~\ref{fig:construction6} we show the Hamiltonian cycle $C_2$ in the graph  $F_{4,2}^{\{2\}}$. 
\begin{figure}[h!]
	\centering
	\includegraphics[width=0.8\textwidth]{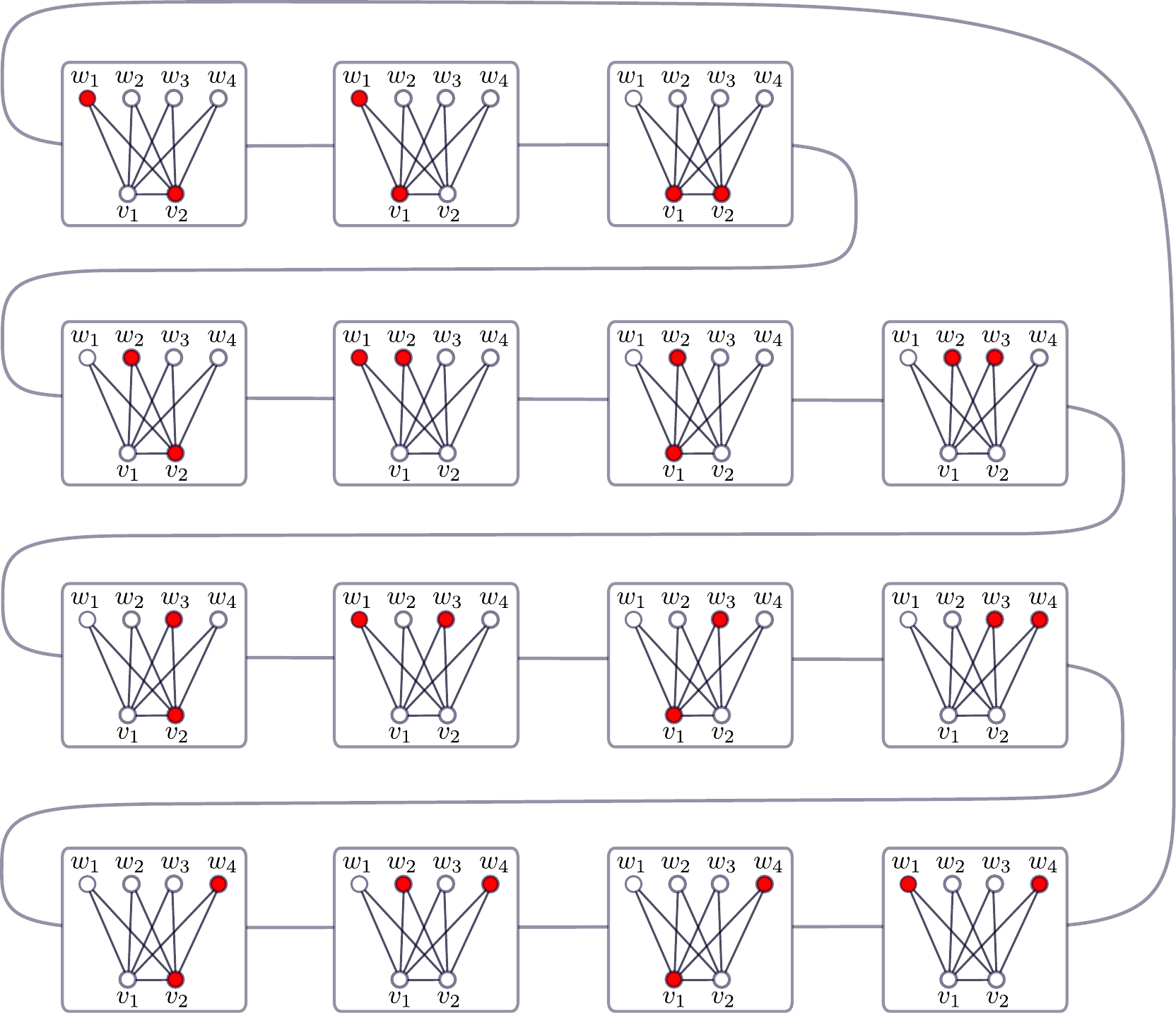}
	\caption{Hamiltonian cycle $C_2$ in the graph $F_{4,2}^{\{2\}}$. }
	\label{fig:construction6}
\end{figure}

\item {\bf Case} $\boldsymbol{1<m<2\,n.}$
	
	Consider again the paths $P_1,\dots, P_m$ defined in the previous case and let us modify them slightly in the following
	way:  
	\begin{itemize}
		\item $P_1'=P_1$;
		\item for $1<i<m$, let $P'_i$ be the path obtained from $P_i$ by deleting the vertices of type $\{w_i, w_j\}$, for each $j>m$;
		\item let $P_m'$ be the path obtained from $P_m$ by first interchanging the vertices $\{w_m, w_{m+1}\}$ and $\{w_m, w_1\}$ from their current positions in $P_m$, and then deleting the vertices of type $\{w_m, w_j\}$, for every $j>m$.
	\end{itemize}
	Given this construction of $P'_i$ we have the following: 
	\begin{itemize}
		\item[(A1)] $P'_i$ induces a path in $F_{m,n}^{\{2\}}$;
		\item[(A2)] for $1\leq i<m$ the path $P'_i$ has the same initial and final vertices as the path $P_i$, and $P'_m$ has the same initial vertex as $P_m$, and its final vertex is $\{w_i, w_1\}$; 
		\item[(A3)] since we have deleted only the vertices of type $\{w_i, w_j\}$ from $P_i$ to obtain $P_i'$, for each $j>m$ and $i \in [m]$, it follows that $\{V(P'_1),\ldots,V(P'_m)\}$ is a partition of $V\left(F_{m,n}^{\{2\}}\right)$. 
		
	\end{itemize}
	
	By (A1) and (A2) we can concatenate the paths $P'_1,\ldots,P'_m$  into a cycle $C'$ as follows:
	\[
	C':=P'_1\,P'_2\,\ldots\,P'_m(v_n, w_1)
	\]
	and then by (A3) it follows that $C'$ is a Hamiltonian cycle in $F_{m,n}^{\{2\}}$. Again, the vertices $\{w_1,v_1\}$ and $\{v_1,v_2\}$ are adjacent in $C'$ since they are adjacent in $P'_1$.

\item {\bf Case} $\boldsymbol{m>2\,n.}$
	
	Here, our aim is to show that $F_{m,n}^{\{2\}}$ is not Hamiltonian by using
	the following known result posed in West's book~\cite{west}. 
	
	\begin{prop}[Prop. 7.2.3, \cite{west}] If $G$ has a Hamiltonian cycle, then for each 
		nonempty set $S\subset V(G)$, the graph $G-S$ has at most $|S|$ connected components.  
	\end{prop}
	
	Then, we are going to exhibit a subset $S\subset V\left(F_{m,n}^{\{2\}}\right)$ such that
	\[\mu\left(F_{m,n}^{\{2\}}-S\right)>|S|.\]
	Let 
	\[
	S:=\big\{\{w_i,v_j\}:i\in [m] \text{ and }j\in [n]\big\}.
	\]
	Note that for each $i,j\in [m]$ with $i\neq j$, the vertex $\{w_i,w_j\}$ has all its neighbors in $S$, and so it is an isolated vertex of $F_{m,n}^{\{2\}}-S$. Since there are $\binom{m}{2}$ vertices of this type, we have this number of connected components of $F_{m,n}^{\{2\}}-S$, each of them having order one. On the other hand, the subgraph induced by the vertices of type $\{v_i,v_j\}$, for $i,j\in [n]$ with $i\neq j$, is a connected component of $F_{m,n}^{\{2\}}-S$. Then, $F_{m,n}^{\{2\}}-S$ has at least $\binom{m}{2}+1$ connected components. Since $m>2n$, it follows that \[\binom{m}{2}=\frac{m(m-1)}{2}\ge \frac{m(2n)}{2}=mn,\] and so
	\[\mu\left(F_{m,n}^{\{2\}}-S\right)\geq \binom{m}{2}+1>mn=|S|,\]
	as required. This completes the proof of Theorem~\ref{thm:main1}. 
\end{itemize}

\section{Proof of Theorem~\ref{thm:general-case-main}}
\label{sec:main}

	Let us first note a well-known property of token graphs that will be used throughout this section. 
	\begin{proposition} 
		\label{prop:complement}
	Let $G$ be a graph of order $n\ge 2$ and let $k$ be an integer with $1\le k\le n-1$. Then $G^{\{k\}}$  is isomorphic to  $G^{\{n-k\}}$.
	\end{proposition}
\begin{proof}
	Let $\psi:V\left(G^{\{k\}}\right)\to V\left(G^{\{n-k\}}\right)$ be the map given by
	\[\psi(A)=V(G)\setminus A.\]
	Clearly $\psi$ is  a bijection. Moreover, for any two vertices $A,B\in V\left(G^{\{k\}}\right)$ we have	
	\[A\triangle B=(V(G)\setminus A)\triangle (V(G)\setminus B)=\psi(A)\triangle \psi(B).\]
	Therefore, $A$ and $B$ are adjacent in $G^{\{k\}}$ if and only if $\psi(A)$ and $\psi(B)$ are adjacent in $G^{\{n-k\}}$. Thus, $\psi$ is an isomorphism. 
\end{proof}

The Hamiltonicity of $F_{1,n}^{\{k\}}$ was proved in~\cite{rive-tru}. However, in order to prove Theorem~\ref{thm:general-case-main}  we need a special Hamiltonian cycle in $F_{1, n}^{\{k\}}$. 

\begin{lemma}
		\label{lemma:m=1}
	Let $n$ and $k$ be integers with $1\leq k\le  n$. Then $F_{1,n}^{\{k\}}$ have a Hamiltonian cycle 
	$C$ in which the vertices $\{w_1,v_1,v_2,\ldots,v_{k-1}\}$ and $\{v_1,v_2,\ldots,v_k\}$
	are adjacent. 
\end{lemma}
\begin{proof}
	
	We proceed by induction on $k$. 
		If $k=1$, then we want a Hamiltonian cycle $C$ in $F_{1,n}^{\{1\}}$ in which the vertices $\{w_1\}$ and $\{v_1\}$ are adjacent in $C$. We define \[C:=\{w_1\}\{v_1\}\{v_2\}\dots \{v_n\},\]
	so, clearly, $C$ holds the lemma. The case $k=2$ corresponds to Theorem~\ref{thm:main1}, and this was done in Section~\ref{sec:double}. Assume from now on that $k>2$. \\	
Assume as induction hypothesis that $F_{1,n'}^{\{k-1\}}$ satisfies the lemma, for all  $n'\ge k-1> 1$. 
	
	For $i$ with $k-1\le i\leq n$, let $H_i$ be the subgraph of $F_{1,n}^{\{k\}}$ induced by the vertex set 
	\[S_i:=\left\{\{v_{j_1},v_{j_2},\ldots,v_{j_k}\}\in V\left(F_{1,n}^{\{k\}}\right) \colon 0\leq j_1< j_2 < \dots < j_k=i\right\}.
	\]
	The subgraph $H_i$ can be understood as the subgraph of $F_{1,n}^{\{k\}}$ induced by all the $k$-token configurations in which there is 
	a token fixed at vertex $v_i$ and the remaining $k-1$ tokens are moving on the subgraph induced by $\{v_0,v_1,\ldots,v_{i-1}\}$
	(which is isomorphic to $F_{1,i-1}$). With this in mind, note that  $H_i\simeq F_{1,i-1}^{\{k-1\}}$. Besides, 	we have the following. 
	\begin{remark}
		\label{rem:partition}
		$\{S_{k-1},S_k,\ldots,S_n\}$ is a partition of $V\left(F_{1,n}^{\{k\}}\right)$.
	\end{remark}
	
	For $i$ with $k+1\le i \le n$, we know that $H_i\simeq F_{1,i-1}^{\{k-1\}}$, then taking $n'=i-1$ we have $n'=i-1\ge k-1>1$. Thus, by induction hypothesis,   
	there is a Hamiltonian cycle 
	$C_i$ in $H_i$ where the vertices \[X_i:=\{v_0,v_1,v_2,\ldots,v_{k-2},v_i\} \text{\ and \ }Y_i:=\{v_1,v_2,\ldots,v_{k-1},v_i\}\]
	are adjacent in $C_i$. Let $P_i$ be the path obtained from $C_i$ by deleting the edge $X_iY_i$. Assume, without loss of generality, that the initial vertex of $P_i$ is $X_i$ and its final vertex is $Y_i$. Thus, we have the following.
	\begin{remark}
		\label{rem:hamiltonianpath}
		For $i$ with $k+1\leq i\leq n$, $P_i$ is a Hamiltonian path of $H_i$.
	\end{remark}
	
	Let us now proceed by cases depending on the parity of $n-k$. 
	\begin{itemize}
		\item \textbf{$\boldsymbol{n-k}$ is odd:}
		
		In this case note that $n-k\ge 1$ and so $n\ge k+1$.
		Consider the vertex sets $S_{k-1}$ and $S_k$. 
		For $0\leq j\leq k$, let \[Z_j:=\{v_0,v_1,\ldots,v_k\}\setminus \{v_j\}.\] 
		Then, $S_{k-1}=\{Z_k\}$ and $S_k=\{Z_0,Z_1,\ldots,Z_{k-1}\}$, 
		with the following adjacencies (in $F_{1,n}^{\{k\}}$): $Z_0$ is adjacent to  $Z_\ell$, for each $\ell$ with $1\le \ell\le k$, and $Z_t$ is adjacent to $Z_{t-1}$, for each $t$ with $1< t \le k$. 
		Let 
		\[P_k:=Z_{k-1}Z_{k-2}\dots Z_1Z_0Z_k.\] 
		By the adjacencies among vertices $Z_0,Z_1,\dots,Z_{k-1}$ we have the following.  
		\begin{remark}
			\label{rem:secondhamiltonianpath}
			$P_k$ is a Hamiltonian path of the subgraph induced by $S_{k-1}\cup S_k$.	
		\end{remark}
		
		Let 
		\[C:= 
		\overleftarrow{P_{k}}\, P_{k+1}\,
		\overleftarrow{P_{k+2}}\,P_{k+3}\,\ldots\,\overleftarrow{P_{n-1}}\,P_{n}.
		\]

			We show that $C$ is a Hamiltonian cycle of $F_{1,n}^{\{k\}}$. Observe the following:
			\begin{enumerate}[label=($\arabic*$)]
				\item the final vertex of $\overleftarrow{P_k}$ is the vertex $Z_{k-1}=\{v_0,v_1,\dots,v_{k-2},v_k\}$, while the initial vertex of $P_{k+1}$ is the vertex $X_{k+1}=\{v_0,v_1,\dots,v_{k-2},v_{k+1}\}$, and these two vertices are adjacent in $F_{1,n}^{\{k\}}$;
				\item for $i$ with $k+1\le i\le n-1$, the final vertex of $P_i$ is the vertex $Y_i=\{v_1,v_2,\dots,v_{k-1},v_i\}$, while the initial vertex of $\overleftarrow{P_{i+1}}$ is the vertex $Y_{i+1}=\{v_1,v_2,\dots,v_{k-1},v_{i+1}\}$, and these two vertices are adjacent in $F_{1,n}^{\{k\}}$; 
				\item for $i$ with $k+1\le i\le n-1$,  the final vertex of $\overleftarrow{P_i}$ is $X_i=\{v_0,v_1,\dots,v_{k-2},v_i\}$, while the initial vertex of $P_{i+1}$ is $X_{i+1}=\{v_0,v_1,\dots,v_{k-2},v_{i+1}\}$, and these two vertices are adjacent in $F_{1,n}^{\{k\}}$; 
				\item finally, the final vertex of $P_n$ is the vertex $Y_n=\{v_1,v_2,\dots,v_{k-1},v_n\}$ while the initial vertex of $\overleftarrow{P_{k}}$ is the vertex $Z_k=\{v_0,v_1,\dots,v_{k-1}\}$, and these two vertices are adjacent in $F_{1,n}^{\{k\}}$.   	
			\end{enumerate}
			Statements (1)--(4) together imply that $C$ is a cycle of $F_{1,n}^{\{k\}}$, and Remarks \ref{rem:partition}, \ref{rem:hamiltonianpath} and \ref{rem:secondhamiltonianpath}
			together imply that the cycle $C$ is Hamiltonian.  
			Finally, note that the vertices $Z_k=\{v_0,v_1,v_2,\ldots,v_{k-1}\}=\{w_1,v_1,v_2,\ldots,v_{k-1}\}$
			and $Z_0=\{v_1,v_2,\ldots,v_k\}$ are adjacent in $C$ (since they are adjacent in $P_k$), as required.

\item \textbf{$\boldsymbol{n-k}$ is even:} 

Suppose first that $n-k=0$, then $n=k$. In this case we have that $\{S_{k-1},S_k\}$ is a partition of $V\left(F_{1,n}^{\{k\}}\right)$. Here, consider the path $P_k$ defined in the previous case as \[P_k=Z_{k-1}Z_{k-2}\dots Z_1Z_0Z_k,\]
where $Z_j=\{v_0,v_1,\dots,v_k\}\setminus \{v_j\}$, for $j\in \{0,1,\dots,k\}$. Then, by the adjacencies among the vertices $Z_0,Z_1,\dots,Z_k$, it follows that $P_k$ induces a cycle in $F_{1,n}^{\{k\}}$, where the vertices $Z_k=\{v_0,v_1,\dots,v_{k-1}\}$ and $Z_0=\{v_1,v_2,\dots,v_k\}$ are adjacent in $P_k$. Since $\{S_{k-1},S_k\}$ is a partition of $V(F_{1,n}^{\{k\}})$, we have that $P_k$ is our desired Hamiltonian cycle. 

Assume $n-k\ge 2$, so $n\ge k+2$. 		
		Let $H$ be the subgraph of $F_{1,n}^{\{k\}}$ induced by the vertex set $S_{k-1}\cup S_k\cup S_{k+1}$. $H$ can be understood as the subgraph induced by all the $k$-token configurations in which the $k$ tokens are placed on $k$ vertices of $\{v_0,v_1,\dots,v_{k+1}\}$. Since the subgraph induced by the vertex set $\{v_0,v_1,\dots,v_{k+1}\}$ is isomorphic to $F_{1,k+1}$, it follows that $H$ is isomorphic to $F_{1,k+1}^{\{k\}}$.  
		
		By Proposition~\ref{prop:complement} we have that $H\simeq F_{1,k+1}^{\{k\}}\simeq F_{1,k+1}^{\{2\}}$.
		We are going to construct a Hamiltonian path $P$ of $H$. 
		
		For $i,j\in \{0,1,\ldots,k+1\}$ with $i\neq j$, let $A_{i,j}=\{v_0,v_1,\ldots,v_{k+1}\}\setminus \{v_i,v_j\}$. Then, 
		two vertices $A_{i,j}$ and $A_{r,t}$ are adjacent in $H$ if and only if $\{v_i,v_j\}$ and $\{v_r,v_t\}$ are adjacent
		in $F_{1,k+1}^{\{2\}}$. 
		
		For $1\leq t\leq k$, let
		\[R_t:=\begin{cases}
			A_{1,k}A_{1,k+1}A_{1,0}A_{1,k-1}A_{1,k-2}A_{1,k-3}\ldots A_{1,2} & \text{ if $t=1$, } \\
			A_{t,0}A_{t,k+1}A_{t,k}A_{t,k-1}A_{t,k-2}\ldots A_{t,t+1} & \text{ if $1<t<k$, and } \\
			A_{k,0} & \text{ if $t=k$. }
		\end{cases}
		\]
		Note that $R_1,R_2,\ldots,R_k$ are paths in $H$, and the concatenation
		$R:=R_1\,R_2\,\ldots R_k$ is a cycle in $H$, where the vertices $A_{k-1,k}$ and $A_{k,0}$ 
		are adjacent in $R$ (since $A_{k-1,k}$ is the final vertex of $R_{k-1}$ and $A_{k,0}$ is the initial vertex of $R_k$). Let $R'$ be the path obtained from $R$ by deleting the edge $(A_{k-1,k},A_{k,0})$, 
		and assume, without loss of generality, that the initial vertex of $R'$ is $A_{k,0}$ and its final 
		vertex is $A_{k-1,k}$. Now, let  
		\[P:=A_{k,k+1}\,A_{0,k+1}\,R'.\]
		We have the following.
		\begin{remark}
			\label{rem:secondpart-hamiltonianpath}
			$P$ is a Hamiltonian path of $H$ with initial vertex $A_{k,k+1}=\{v_0,\ldots,v_{k-1}\}$
			and final vertex $A_{k-1,k}=\{v_0,\ldots,v_{k-2},v_{k+1}\}$. Moreover, $A_{0,k+1}=\{v_1,v_2,\dots,v_k\}$ and $A_{k,k+1}=\{v_0,v_1,\dots,v_{k-1}\}$ are adjacent in $P$. 
		\end{remark}  
		
		Here we use the paths $P_{k+2},P_{k+3},\ldots,P_n$ defined above. 
		Let 
		\[C:=P\, P_{k+2}\,\overleftarrow{P_{k+3}}\,P_{k+4}\,\dots\,\overleftarrow{P_{n-1}}\,P_n.\]
		Proceeding similarly to the previous case it can be shown that $C$ is a cycle. 
 
 Remarks \ref{rem:partition}, \ref{rem:hamiltonianpath} and \ref{rem:secondpart-hamiltonianpath}
		together imply that $C$ is a Hamiltonian cycle of $F_{1,n}^{\{k\}}$. Finally, by Remark~\ref{rem:secondpart-hamiltonianpath} we know that vertices 
		$A_{k,k+1}=\{v_0,v_1,\ldots,v_{k-1}\}=\{w_1,v_1,\ldots,v_{k-1}\}$ and $A_{0,k+1}=\{v_1,v_2,\ldots,v_k\}$
		are adjacent in $C$, as required. 
	\end{itemize}
	Thus, in both cases, we have our desired Hamiltonian cycle. 
\end{proof}

In Figure~\ref{fig:construction5} is shown a Hamiltonian cycle of $F_{1,5}^{\{3\}}$ constructed as in the proof of Lemma~\ref{lemma:m=1}. 

\begin{figure}[h!]
	\centering
	\includegraphics[width=0.7\textwidth]{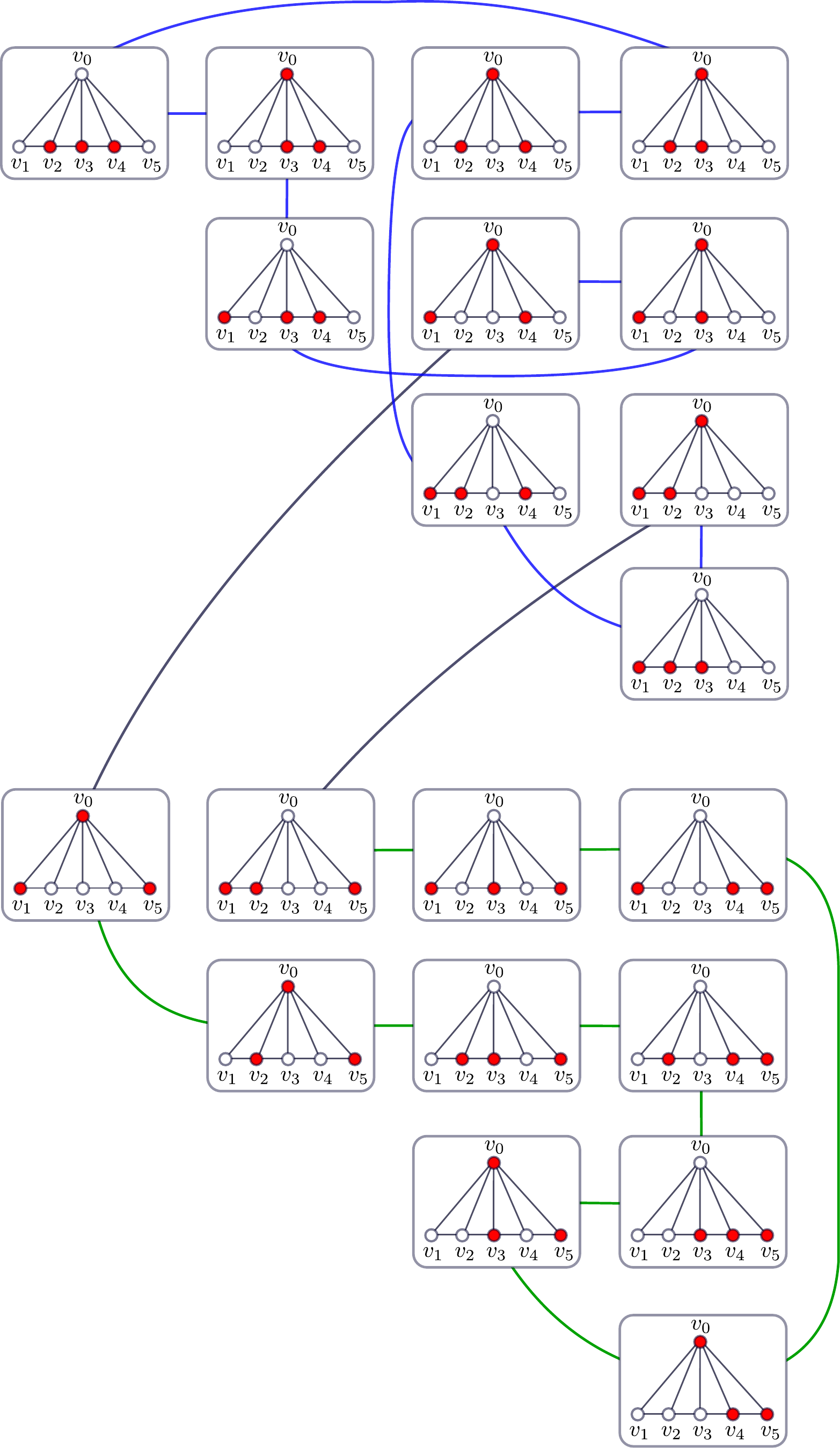}
	\caption{Hamiltonian cycle in $F_{1,5}^{\{3\}}$. }
	\label{fig:construction5}
\end{figure}

Now, we prove our main result for the fan graph $F_{m,n}$, with $m\geq 1$. We recall it next. 

\kgeneral*  

\begin{proof}[Proof of Theorem~\ref{thm:general-case-main}] 
	We claim that $F_{m,n}^{\{k\}}$ has a Hamiltonian cycle $C$ in which the vertices 
	\[\{w_1,v_1,v_2,\ldots,v_{k-1}\} \text{ \ and \ }\{v_1,v_2,\ldots,v_k\}\] are adjacent.  
	Clearly, this claim implies the theorem. 
	To show the claim, we use induction on $m$. The case $m=1$ is proved in Lemma~\ref{lemma:m=1}.  

 Assume as induction hypothesis that the claim holds for all $F_{m-1,n}^{\{k'\}}$ with $2\le k'\le n$ and $1\le m-1\le 2n$.
 
 	Note that the case $k=2$ corresponds to Theorem~\ref{thm:main1}, so that we assume $k>2$.
 
	Let 
	\[S_1:=\{A\in F_{m,n}^{\{k\}} \colon w_1\in A\} \text{ \ and \ } S_2:=\{B\in F_{m,n}^{\{k\}} \colon w_1\notin B\}.\]
	Clearly, $\{S_1,S_2\}$ is a partition of $V\left(F_{m,n}^{\{k\}}\right)$.
	Let $H_1$ and $H_2$ be the subgraphs of $F_{m,n}^{\{k\}}$ induced by $S_1$ and $S_2$, respectively. 

Note that $H_1\simeq F_{m-1,n}^{\{k-1\}}$ and $H_2\simeq F_{m-1,n}^{\{k\}}$. 
	By the induction hypothesis, there are cycles $C_1$ and $C_2$ such that
	\begin{enumerate}[label=(\roman*)]
		\item $C_1$ is a Hamiltonian cycle of $H_1$, where the vertices 
		$X_1:=\{w_1,w_{2},v_1,v_2,\ldots,v_{k-2}\}$ and $Y_1:=\{w_1,v_1,v_2,\ldots,v_{k-1}\}$
		are adjacent in $C_1$; and 
		\item $C_2$ is a Hamiltonian cycle of $H_2$, where the vertices
		$X_2:=\{w_{2},v_1,v_2,\ldots,v_{k-1}\}$ and $Y_2:=\{v_1,v_2,\ldots,v_k\}$
		are adjacent in $C_2$.  
	\end{enumerate}
	For $i=1,2$, let $P_i$ be the subpath of $C_i$, obtained by deleting the edge 
	$X_iY_i$. Note that $P_i$ is a Hamiltonian path of $H_i$ joining the vertices $X_i$ and $Y_i$, and let us assume that the initial vertex of $P_i$ is $X_i$ and its final vertex is $Y_i$. 
	On the other hand, note that $X_1$ is adjacent to $X_2$ and $Y_1$ is adjacent to $Y_2$, these two facts together imply 
	that the concatenation 
	\[C:=P_1\,\overleftarrow{P_2} \] 
	is a cycle in $F_{m,n}^{\{k\}}$. 
	Since $\{S_1,S_2\}$ is a partition of $V\left(F_{m,n}^{\{k\}}\right)$, it follows that $C$ is  
	Hamiltonian. Finally, note that the vertices 
	\[Y_1=\{w_1,v_1,v_2,\ldots,v_{k-1}\}\text{\ and \ }Y_2=\{v_1,v_2,\ldots,v_k\}\] are adjacent in $C$, since $Y_1$ is the final vertex of $P_1$ and $Y_2$ is the initial vertex of $\overleftarrow{P_2}$. This completes the proof.	
\end{proof}

\section{A relationship between Gray codes for combinations and the Hamiltonicity of token graphs}
\label{app:Gray}

There are several applications of token graphs to Physics and Coding Theory; see, e.g.~\cite{fisch2, fisch, Jo, soto}. 
 
Regarding the Hamiltonicity, there is a direct relationship between the Hamiltonicity of token graphs and Gray codes for combinations. 

Consider the problem of generating all the subsets of an $n$-set, which can be reduced to the problem of generating all possible binary strings of length $n$ (since each $k$-subset can be transformed into a $n$-binary string by placing an $1$ in the $j$-th entry if $j$ belongs to the subset, and $0$ otherwise). The most straightforward way of generating all these $n$-binary strings is counting in binary; however, many elements may change from one string to the next. Thus, it is desirable that only a few elements change between successive
strings. The case when successive strings differ by a single bit, is commonly known as \emph{Gray codes}. Similarly, the problem of generating all the $k$-subsets of an $n$-set is reduced to the problem of generating all the $n$-binary strings of constant weight $k$ (with exactly $k$ $1$'s). 

Gray codes are known to have applications in different areas, such as cryptography, circuit testing, statistics and exhaustive combinatorial searches. For a more detailed information on Gray codes, we refer the reader to \cite{baylis, ruskey-graycodes, savage}. 
Next, we present a formal definition of Gray codes. 

Let $S$ be a set of $n$ combinatorial objects and $C$ a relation on $S$, $C$ is called the \textit{closeness relation}. A \emph{Combinatorial Gray Code} (or simply \emph{Gray code}) for $S$ is a sequence $s_1,s_2,\ldots,s_n$ of the elements of $S$ such that $(s_i,s_{i+1})\in C$, for $i=1,2,\ldots,n-1$. Additionally, if $(s_n,s_1)\in C$, the Gray code is said to be \emph{cyclic}.
In other words, a Gray code for $S$ with respect to $C$ is a listing of the elements of $S$ in which successive elements are close (with respect to $C$).  There is a digraph $G(S,C)$, the \emph{closeness graph}, associated to $S$ with respect to $C$, where the vertex set and edge set of $G(S,C)$ are $S$ and $C$, respectively. If the closeness relation is symmetric, $G(S,C)$ is an undirected graph. A Gray code (resp. cyclic Gray code) for $S$ with respect to $C$ is a Hamiltonian path (resp. a Hamiltonian cycle) in $G(S,C)$. 

We are interested in Gray codes for combinations. A $k$-combination of the set $[n]$ is a $k$-subset of $[n]$, which in turn, can be thought as a binary string of lenght $n$ and constant weight $k$ (it has $k$ $1$'s and $n-k$ $0$'s). Consider the set $S=S(n,k)$ of all the $k$-combinations of $[n]$. Next, we mention three closeness relations that can be applied to $S$; for other closeness relations we refer to \cite{ruskey-graycodes}. 
\begin{itemize}
	\item[1)] The \emph{transposition} condition: two $k$-subsets are close if they differ in exactly two elements. Example: $\{1,2,5\}$ and $\{2,4,5\}$ are close, while $\{1,2,5\}$ and $\{1,3,4\}$ are not. 
	\item[2)] The \emph{adjacent transposition} condition: two $k$-subsets are close if they differ in exactly two consecutive elements $i$ and $i+1$. Example: $\{1,2,5\}$ and $\{1,3,5\}$ are close, while $\{1,2,5\}$ and $\{1,4,5\}$ are not. 
	\item[3)] The \emph{one or two apart transposition} condition: two $k$-subsets are close if they differ in exactly two elements $i$ and $j$, with $|i-j|\leq 2$. Example: $\{1,2,5\}$ and $\{1,4,5\}$ are close, while $\{1,2,5\}$ and $\{2,4,5\}$ are not.  
\end{itemize}

The relationship between the closeness graph associated to $S$ with respect to these closeness conditions and some token graphs are showed in the following propositions.  

\begin{proposition}
	\label{prop:transp}
	If $C$ is the transposition condition, then the closeness graph $G(S,C)$ is isomorphic to $K_n^{\{k\}}$, where $K_n$ denotes the complete graph of order $n$. 
\end{proposition}
\begin{proof}
	Let $V(K_n):=\{1,2,\dots,n\}$. Then, $S=V(K_n^{\{k\}})$. Note that two $k$-subsets $A$ and $B$ are adjacent in $G(S,C)$ if and only if they differ in exactly two elements, that is, $|A\triangle B|=2$, and this holds if and only if $A$ and $B$ (as vertices of $K_n^{\{k\}}$) are adjacent in $K_n^{\{k\}}$. 
\end{proof}

\begin{proposition}
	\label{prop:adjac}
	If $C$ is the adjacent transposition condition, then the closeness graph $G(S,C)$ is isomorphic to $P_n^{\{k\}}$, where $P_n$ denotes the path graph of order $n$. 
\end{proposition}
\begin{proof}
	Let $V(P_n):=\{1,2,\dots,n\}$ with $i$ and $i+1$ adjacent for each $i\in [n-1]$. Then $S=V(P_n^{\{k\}})$. Two $k$-subsets $A$ and $B$ are adjacent in $G(S,C)$ if and only if they differ in exactly two consecutive elements $i$ and $i+1$, that is, $A\triangle B=\{i,i+1\}$, and this holds if and only if $A$ and $B$ are adjacent in $P_n^{\{k\}}$. 
\end{proof}

Given a graph $G$, the \emph{square} $G^2$ of $G$ is the graph on $V(G)$ in which two vertices are adjacent in $G^2$ if they are at distance at most two in $G$. 

\begin{proposition}
	\label{prop:apart}
	If $C$ is the one or two apart transposition condition, then the closeness graph $G(S,C)$ is isomorphic to $\left(P_n^2\right)^{\{k\}}$, where $P_n^2$ denotes the square of the path graph of order $n$. 
\end{proposition}
\begin{proof} Let $V(P_n):=\{1,2,\dots,n\}$ with $i$ and $j$ adjacent whenever $|i-j|\le 2$. Then $S=V((P_n)^{\{k\}})$. Any two $k$-subsets $A$ and $B$ are adjacent in $G(S,C)$ if and only if they differ in exactly two elements $i$ and $j$ with $|i-j|\le 2$, that is, $A\triangle B=\{i,j\}$ with $i\neq j$ and $|i-j|\le 2$,  and this holds if and only if $A$ and $B$ are adjacent in $(P_n)^{\{k\}}$. 	
\end{proof}

Note that when the closeness graph $G(S,C)$ associated to $S$ with respect to $C$ is isomorphic to the $k$-token graph of some graph $G$, then  a Gray code and a cyclic Gray code for $S$ with respect to $C$ correspond to a Hamiltonian path and a Hamiltonian cycle, respectively, of the $k$-token graph of $G$.

\section{Conclusions}
\label{sec:concl}

The study of token graphs began in the 90's, and since then, several connections of token graphs with other research areas have been discovered, such as Quantum Mechanics and Coding Theory; several researchers are currently exploring more applications of token graphs to Physics. The study of token graphs is a current line of research for several researchers along the world. In this paper we study the Hamiltonicity of token graphs of join graphs. Our main results are the following: we provide necessary and sufficient conditions for the Hamiltonicity of the 2-token graphs of fan graphs $F_{m,n}$. For the $k$-token graph of the fan graph $F_{m,n}$, we provide sufficient conditions on the parameters $m$, $n$ and $k$, with $2< k < m+n-2$, for having $F_{m,n}^{\{k\}}$ Hamiltonian. 
As a corollary of this last result, and using a simple property of token graphs, we provide sufficient conditions on the graphs $G_1$ and $G_2$ such that the $k$-token graph of the join graph $G=G_1+G_2$ is Hamiltonian. Our results provides, to our knowledge, the first family of non-Hamiltonian graphs $G$ with Hamiltonian $k$-token graphs $G^{\{k\}}$, where $2<k<|G|-2$. Before this work, only a finite number of graphs satisfying this property were known. 

There are several open lines of research regarding the Hamiltonicity of token graphs. We would like to suggest some open problems for future research. 

\begin{itemize}
	\item[1.]  To find other families of graphs with Hamiltonian $k$-token graphs. 
\end{itemize}

\begin{itemize}
	\item[2.] Given two graphs $G$ and $H$, consider the Cartesian product $G\, \square\, H$ of $G$ and $H$. To study the Hamiltonicity of $(G\,\square\, H)^{\{k\}}$
	in terms of the Hamiltonicity of $G$ and $H$. Similarly for other products of graphs as the corona of two graphs. 
\end{itemize}
For $k=2$, it is known that the smallest Hamiltonian graph $G$ for which $G^{\{2\}}$ is Hamiltonian, is a cycle with an odd chord. Then, a natural problem is the following. 
\begin{itemize}
	\item[3.] For $k>2$, to find the smallest Hamiltonian graph $G$ for which $G^{\{k\}}$ is Hamiltonian. 
\end{itemize}

\section{Acknowledgements}

We thank the referees for their useful suggestions. A. L. Trujillo-Negrete was partially supported by CONACYT (Mexico), grant 253261. 


%

\end{document}